\documentclass[12pt]{article}

\usepackage[top=1in, bottom=1in, left=1in, right=1in]{geometry}

\usepackage[colorlinks=false,urlcolor=blue,citecolor=blue,linkcolor=blue,
bookmarks=true,
     bookmarksopen=false,pdftitle=created by dvipdf,pdfcreator=NM,
     pdfauthor=Megiddo,pdfsubject=mor.sty:6.29.04]{hyperref}
\usepackage{url}
\usepackage{subfigure}
\usepackage{amsfonts,mathrsfs,amsthm}
\usepackage{amssymb,amsmath}
\usepackage{verbatim}
\usepackage{acronym}
\usepackage{graphicx}
\usepackage{enumerate}
\usepackage[usenames]{color}

\def\beh#1{{\color{black}#1}}

\def\fskip#1{}

\newtheorem{theorem}{Theorem}

\newtheorem{corollary}{Corollary}

\newtheorem{definition}{Definition}

\newtheorem{proposition}[theorem]{Proposition}

\newcommand{\Einf}{\mathcal{E}^{\infty}}

\def\1{{\bf 1}}

\newcommand{\al}{\alpha}

\def\N{\mathcal{N}}

\def\R{\mathbb{R}}

\newcommand{\EXP}[1]{\mathsf{E}\!\left[#1\right] }
\newcommand{\prob}[1]{\mathsf{Pr}\left( #1 \right)}
\newcommand{\Ginf}{G^{\infty}}

\newcommand{\bfo}{\mathbf{1}}

\newcommand{\F}{\mathcal{F}}
\newcommand{\Fc}{\{\mathcal{F}_k\}}
\newcommand{\Fk}{\mathcal{F}_k}

\newcommand{\M}{\mathcal{M}}
\newcommand{\Pset}{\mathscr{P}}
\newcommand{\Sc}{\{S(k)\}}

\newcommand{\Wc}{\{W(k)\}}
\newcommand{\Wbs}{W_{\bar{S}(k+1)S(k)}(k+1)}
\newcommand{\Ws}{W_{{S}(k+1)\bar{S}(k)}(k+1)}
\newcommand{\xc}{\{x(k)\}}
\newcommand{\Zp}{\mathbb{Z}^+}

\begin{document}
\title{On Endogenous Random Consensus and Averaging Dynamics}
\author{Behrouz Touri and Cedric Langbort\thanks{Behrouz Touri is with the School of Electrical and Computer Engineering at Georgia Tech University and Cedric Langbort is with the Department of Aerospace Engineering at University of
Illinois, Email: touri@gatech.edu, langbort@illinois.edu.
This research is supported in parts by NSF Career Grant 11-51076 and Air Force MURI Grant FA9550-10-1-0573.}
}
\maketitle
\begin{abstract}
 Motivated by various random variations of Hegselmann-Krause model for opinion dynamics and gossip algorithm in an endogenously changing environment, we propose a general framework for the study of endogenously varying random averaging dynamics, i.e.\ an averaging dynamics whose evolution suffers from history dependent sources of randomness. We show that under general assumptions on the averaging dynamics, such dynamics is convergent almost surely. We also determine the limiting behavior of such dynamics and show such dynamics admit infinitely many time-varying Lyapunov functions.
\end{abstract}
\section{Introduction}
 \beh{In this work we study random averaging dynamics, i.e.\ dynamics of the form $x(k+1)=W(k+1)x(k)$, where $W(k+1)$ is a random non-negative $m\times m$ matrix whose rows sum up to one and $\{x(k)\}$ is a random vector process evolving in $\R^m$.} These dynamics are one of the fundamental tools in the theory of time-varying Markov chains \cite{hajnal,hajnal1958weak}, distributed computation \cite{Tsitsiklis84}, distributed optimization \cite{Tsitsiklis84,Nedic09,NedichParilloOzdaglar,nedic2013distributed},  distributed estimation \cite{KarMoura07, khan2008distributing}, distributed rendezvous \cite{cortes2006robust}, and opinion dynamics \cite{Krause:02}.

 Until recently, most of the research in this domain has been focused on the study of those dynamics in deterministic settings  \cite{Tsitsiklis86,Jadbabaie03,Xiao04,CaoMora,Olshevsky09a,alex_john_cdc, touri2012backward,bolouki2012}. The random
setting has also been the subject of much attention lately, due to its relevance in many practical applications \cite{matei2008almost,Zampieri,Tahbaz-Salehi08,Tahbaz-Salehi09,ErgodicityPaper,CDC2010,TouriNedich:Approx,TouriNedich:productrandom,LLC:11}. These studies are closely related to the investigation of products of random stochastic matrices as well as to the theory of Markov chains in random environments \cite{Cogburn,Nawrotzki1,Nawrotzki2}. A common feature in most of the past research in this domain is that the averaging process is affected by an exogenous disturbances which does not depend on the history of the process but rather is imposed on the process by some external source of error. Notable exceptions among these works are \cite{LLC:11,lu2011consensus} where the authors developed some sufficient conditions for convergence to consensus of random adapted averaging dynamics. However, as will be discussed later, many random averaging dynamics may not satisfy such conditions (and may not even converge to consensus) even though they are stable (random) dynamics.

  In this work we study products of random stochastic matrices (or equivalently random averaging dynamics) for adapted processes. We show that many of the known results for deterministic dynamics \cite{bolouki2012,BalancedHenderickx} and independent random processes \cite{TouriNedich:productrandom} can be extended to a more general setting of adapted processes. In particular, we show that the main convergence results in \cite{bolouki2012,BalancedHenderickx,TouriNedich:productrandom} hold under mild conditions on the conditional expectation of an adapted random process. As a result, we show that many random variants of the Hegselmann-Krause model for opinion dynamics, such as asynchronous Hegselmann-Krause dynamics, are stable. To the best of our knowledge, none of the previously known results and techniques can address the stability of such dynamics.

   The structure of this paper is as follows: in Section~\ref{sec:motivation}, we motivate our study by some problems from computer science and social sciences. In Section~\ref{sec:analysis}, we set the mathematical framework for the current study and discuss the main results of this work. We discuss the proof of the main results in Section~\ref{sec:proof} which can be skipped by the readers interested in the implications of the main theorems, \beh{which are presented in Section~\ref{sec:implicaitons}.} Finally, in Section~\ref{sec:conc}, we conclude our paper.

 \section{Motivation}\label{sec:motivation}
  Before proceeding with the main technical contributions and setup of our paper, let us discuss some motivating problems namely, Hegselmann-Krause model for opinion dynamics and endogenous random gossiping, .

\subsection{Hegselmann-Krause Dynamics with Random Confidence Levels:} In \cite{Krause:02}, a mathematical model for opinion dynamics in a society is given which is commonly referred to as Hegselmann-Krause dynamics. The motivation there is to model how the opinion of people in a given society is changing with time as a result of their interaction with other agents. In this model, each agent in the society $[m]=\{1,\ldots,m\}$ is assumed to have an initial belief about an issue which is assumed to be representable by a scalar. Thus at time $0$, an agent $i\in[m]$ has an initial opinion $x_i(0)\in \R$. From this time onward, each agent averages out her opinion with agents with similar beliefs in the society. More precisely, let $\N_i(x,\epsilon)=\{j\in [m]\mid \|x_i-x_j\|\leq \epsilon\}$. Then,
       \begin{align}\label{eqn:HKdyn}
         x_i(k+1)=\sum_{j\in \N_i(x(k),\epsilon)}\frac{1}{|\N_i(x(k),\epsilon)|}x_j(k).
       \end{align}
       In this model, $x_i(k)$ is referred to as the opinion of the $i$th agent at time $k$ and the vector $x(k)$ is referred to as the opinion profile of the society at time $k$. This model, later, inspired a distributed rendezvous algorithm for a robotic network \cite{cortes2006robust}. In this model, $\epsilon>0$ is called the confidence level and is assumed to be fixed and homogeneous. Using the tools developed in \cite{Lorenz}
       and \cite{TouriNedicCDC2011}, such dynamics can be analyzed when the confidence level is known deterministically and is not agent dependent.
        Now, consider the following three random variations of the Hegselmann-Krause dynamics which share a common property that the randomness affecting the dynamics is not caused by an independently identically distributed (i.i.d) or stationary process and that the randomness depends on the current states of the agents.
:
        \subsubsection{Asynchronous Hegselmann-Krause Dynamics} In the original formulation of Hegselmann-Krause model, it is assumed that the agents update their states synchronously, i.e.\ all the agents simultaneously update their values using the dynamics \eqref{eqn:HKdyn}. One can consider the asynchronous version of the Hegselmann-Krause dynamics: suppose at time $k\geq 0$, nature picks a random agent $i(k)\in [m]$ and the agent $i(k)$ updates her value according to \eqref{eqn:HKdyn} and the values of the rest of the agents remain unchanged. Due to the asymmetric nature of the sample paths of such dynamics (i.e.\ while agent $i(k)$ averages her belief with her neighbors, none of her neighbors average their opinion with her), none of the known tools to study Hegselmann-Krause dynamics \cite{Lorenz,TouriNedicCDC2011} can address the stability of such dynamics. The main result of this paper implies that, as long as the probability of choosing each agent is uniformly bounded from below by some $\underline{p}>0$, i.e.\ $\prob{i(k)=j\mid \F_k}\geq\underline{p}$ for all $j\in [m]$ and $k\geq 0$, the asynchronous Hegselmann-Krause dynamics converges almost surely (here $\F_k$ is the $\sigma$-algebra generated by $x(0),\ldots, x(k)$).

        \subsubsection{Hegselmann-Krause Dynamics in the Presence of Link-Failure} Consider the Hegselmann-Krause dynamics as described above. Suppose that at each time instant $k\geq 0$, each link between the agents is broken with some probability $p_k\in [0,1]$, i.e.\ at time $k\geq 0$, an agent $j\in \N_i(x(k),\epsilon)$ is removed from this set with probability $p_k$ ($j\not=i$). Note that, when the link-failures are happening independently, there is a possibility that agent $j$ is removed from the set $\N_i(x(k),\epsilon)$ even though $i$ is not removed from $\N_j(x(k),\epsilon)$. One of the implications of the results developed in this work is that the Hegselmann-Krause dynamics converges in the presence of arbitrary link-failure.

        \subsubsection{Hegselmann-Krause Dynamics with Random Confidence Intervals} Consider the original Hegselmann-Krause dynamics as discussed above. Suppose that at time $k$, the confidence level of each agent is drawn from a distribution $\mathcal{E}(k)$  independent of the other agents.  As before, if the confidence level of the agents in this society are drawn randomly and independently, there is a chance that agent $i$ observes agent $j$ in its neighborhood while agent $i$ does not belong to the neighborhood of agent $j$. Using our main results, we will argue that regardless of the distribution of the random confidence interval, such dynamics is stable and convergent almost surely.
 \subsection{Asymmetric Endogenous Gossiping}
 Here, we discuss an extension of the asymmetric gossip algorithm introduced in \cite{fagnani2008asymmetric} to time-varying networks. In the original gossip algorithm \cite{kempe2003gossip,Boyd06}, we have a set $[m]$ of $m$ agents and each of them has an initial scalar $x_i(0)$. At each discrete time instant $k=0,1,2,\ldots$, nature picks two agents $i,j$ with some probability $P_{ij}>0$ from a connected graph $G=([m],E)$ (i.e.\ $\{i,j\}\in E$). Then the two agents set:
 \begin{align}\label{eqn:basicgossip}
   x_i(k+1)=x_j(k+1)=\frac{1}{2}(x_i(k)+x_j(k)),
 \end{align}
 and the value of the other agents remain unchanged, i.e. $x_\ell(k+1)=x_\ell(k)$ for $k\not=i,j$. Now, suppose that at each time instance $k\geq 0$, nature picks an ordered pair of agents $(i(k),j(k))$ randomly (and possibly dependent on the history of the process and her earlier choice). One interesting case of such a choice (e.g.\ in multi-hop wireless network) is when nature picks $i(k)\in[m]$ uniformly and then, it picks $j(k)\in \N_{i}(x(k),\epsilon)$ uniformly, where $\N_{i}(x(k),\epsilon)$ is the agents with similar belief to agent $i$ (as defined in the previous subsection). Then, agent $i(k)$ sends her value $x_i(k)$ to agent $j(k)$ and agent $j(k)$ updates her value as
 \begin{align}\label{eqn:gossipextended}
   x_{j(k)}(k+1)=(1-\gamma(k))x_{j(k)}(k)+\gamma(k) x_{i(k)}(k)),
 \end{align}
 where $\gamma(k)$ is a random variable with support in $[l,h]$ for $0<l\leq h<1$.

  Again due to the endogenous random nature of this dynamics, none of the previously known analysis technique applies here. In Section~\ref{sec:analysis}, we show that if for all $k\geq 0$ and $i,j\in[m]$
  \begin{align}\label{eqn:gossipprob}
    &\prob{(i(k),j(k))=(i,j)\mid \F_k}\cr
    &\qquad\geq \alpha \prob{(i(k),j(k))=(j,i)\mid \F_k},
  \end{align}
  where $\F_k$ is the history of the random evolution up to time $k$, \beh{$\alpha\in(0,1)$ is a constant,} and $\gamma(k)$ is independent of $(i(k),j(k))$, then the random dynamics \eqref{eqn:gossipextended} is convergent almost surely.

   An instance of this dynamics is studied in \cite{Mostagir2013}. It is not hard to see that the model studied in \cite{Mostagir2013} satisfies \eqref{eqn:gossipprob}.
\section{Averaging Dynamics for Adapted Processes}\label{sec:analysis}
 In this section, we present the main result of this paper. We start our discussion by reviewing some notations  that will be used throughout the rest of this paper. Then, we present the main results of this paper and discuss their implications for the study of the random dynamics discussed above.

\beh{Let $(\Omega,\M,\prob{\cdot})$ be a probability space. Also for any $k\in \Zp$, let $W(k):\Omega\to S^{m}$ be a measurable random stochastic matrix where $S^{m}$ is the set of stochastic matrices in $\R^{m\times m}$ (which are non-negative matrices with the property that the entries in each row add up to one). We refer to such a sequence of random matrices as a \textit{random stochastic matrix process}. Finally, for a matrix $W$ and non-trivial subsets $S,T\subset [m]$ (i.e. $S,T\not=\emptyset$ and $S,T\not=[m]$), let:
\[W_{ST}=\sum_{i\in S}\sum_{j\in T}W_{ij}.\]}

 Our main focus in this paper is the dynamics
 \begin{align}\label{eqn:dynsys}
   x(k+1)=W(k+1)x(k),\quad \mbox{for $k\geq 0$},
 \end{align}
 where $x(0):\Omega\to\R^m$ is a random vector (i.e.\ each entry of $x(0)$ is measurable with respect to $\M$), and provide a convergence result for such dynamics. \beh{We refer to such dynamics as \textit{random averaging dynamics}.} As highlighted in \cite{BalancedHenderickx}, \cite{TouriNedicCDC2011}, \cite{bolouki2012}, for different models and types of averaging dynamics, the main idea involved in the proof of convergence is that there is a \textit{balancedness}  between nodes in averaging dynamics. Our goal is to extend those results to history-dependent random processes.\beh{ Motivated by the study in \cite{TouriNedich:productrandom}, we say that a random stochastic matrix process $\Wc$ is balanced if
 \begin{align}\label{eqn:stationarybalanced}
    &\EXP{W_{\bar{S}S}(k+1)\mid W(k),\ldots,W(1),x(0)}\\\nonumber
    &\qquad\geq \alpha \EXP{W_{S\bar{S}}(k+1)\mid W(k),\ldots,W(1),x(0)},
 \end{align}
 for any non-trivial $S\subset [m]$ and some $\alpha>0$ and any $k\geq 0$, where $\bar{S}=[m]\setminus S$ is the complement of the set $S$ (with respect to $[m]$). We refer to $\al$ as the balancedness coefficient.}

 Also, following \cite{CDC2010} for any random stochastic matrix process $\Wc$, let us define the random undirected graph $\Ginf=([m],\Einf)$ to be the graph with the edge set
 \[\Einf(\omega)=\{\{i,j\}\mid \sum_{k=0}^\infty(W_{ij}(k,\omega)+W_{ji}(k,\omega))=\infty\}.\]
  We refer to this graph as the infinite flow graph of the process $\Wc$.

  The first main result of this paper characterizes the convergence and stability properties of a broad class of adapted random averaging dynamics.
  \begin{theorem}\label{thrm:main}
    For any balanced adapted random stochastic matrix process $\Wc$, such that for all $i\in [m]$ and $k\geq 0$, we have $W_{ii}(k)\geq \gamma>0$ almost surely, the dynamics \eqref{eqn:dynsys} converges almost surely. Moreover, $\lim_{k\to\infty}x_i(k,\omega)=\lim_{k\to\infty}x_j(k,\omega)$ if and only if $i,j$ belong to the same connected component of $\Ginf(\omega)$.
  \end{theorem}

  As we will show later in Section~\ref{sec:implicaitons}, convergence and stability of all the dynamics discussed in Section~\ref{sec:motivation} follows immediately from Theorem~\ref{thrm:main}. However, Theorem~\ref{thrm:main} alone does not provide any insight into the rate of convergence to an equilibrium. Our next result shows the existence of infinitely many (stochastic) Lyapunov functions for the study of such systems. As in the case of independent processes \cite{TouriNedich:productrandom}, we show that for any convex function $g$, there exists a Lyapunov function adapted to $g$ for processes portrayed in Theorem~\ref{thrm:main}.
   \begin{corollary}\label{lemma:mainlyapunov}
     For any balanced adapted random stochastic matrix process $\Wc$ satisfying the assumptions of Theorem~\ref{thrm:main}, there exists an adapted random stochastic vector process $\{\bar{\pi}(k)\}$ such that for any convex function $g:\R\to\R$, the random process $\{V_k\}$ defined by
     \[V_k=\sum_{i=1}^m\bar{\pi}_i(k)g(x_i(k))-g(\bar{\pi}^T(k)x(k)),\]
     is a super-martingale and hence, convergent, almost surely.
   \end{corollary}

   In other words, if we look at the random process $\{V_k\}$, we always have 
   \[\EXP{V_{k+1}\mid W(k),\ldots,W(1),x(0)}\leq V_k.\]
    Note that the choice of convex function $g$ is arbitrary in the above theorem. For example for the case of $g(t)=t^2$, we have $V_k=\sum_{i=1}^m\bar{\pi}_i(k)x^2_i(k)-(\bar{\pi}^T(k)x(k))^2$ which is the (random) empirical variance of $x(k)$ with respect to the random probability distribution $\bar{\pi}(k)$. This results opens up many doors to study such dynamics using different Lyapunov functions.

 \section{Proof of the Main Theorem}\label{sec:proof}
  In this section, we provide a proof of Theorem~\ref{thrm:main}. Before discussing the proof, let us introduce some notations that will be used subsequently. \beh{For notational simplicity, we present the proof of the main results based on filtration formalism, i.e.\ we use notion of conditional expectation $\EXP{\cdot\mid \F_k}$ instead of $\EXP{\cdot\mid W(k),\ldots, W(1),x(0)}$, where $\F_k$ is the smallest sub-$\sigma$-algebra of $(\Omega,\M)$ such that $x(0),W(1),\ldots,W(k)$ are measurable with respect to it.  So, let $\Fc$ be such a filtration for $(\Omega,\M)$. Indeed, all the following discussion follows for any filtration for $x(0),W(1),W(2),\ldots,$. With a slight abuse of notation we say that $\Wc$ is an adapted processes to $\Fc$ if $W(k)$ is measurable with respect to $\F_k$ and $x(0)$ is measurable with respect to $\F_0$. } We say that a mapping $S:\Omega\to \Pset([m])$ is a random subset of $[m]=\{1,\ldots,m\}$ if $S$ is measurable with respect to $([m],\Pset([m])$ where $\Pset([m])$ is the set of all subsets of $[m]$. Moreover, we say that a sequence $\{S(k)\}$ of random subsets is adapted to $\Fc$ if $S(k)$ is measurable with respect to $\Fk$.

In our development an object, \textit{regular sequence}, plays a central role. As defined in \cite{TouriNedich:Doubly} in the deterministic setting, a sequence $\{S(k)\}$ of subsets of $[m]$ is called regular if $|S(k)|=|S(0)|\geq 1$ for all $k$, i.e.\ the cardinality of $S(k)$ does not change with time. We say that $\{S(k)\}$ is an adapted regular sequence if $S(k)$ is adapted in the sense above and also $|S(k)|=\ell$ almost surely for some $\ell\in [m]$. It should be clear that in a deterministic setting, i.e.\ the case that $\Fk=\{\emptyset,\Omega\}$ for all $k\geq 0$, the two definitions coincide.

Now, let us define a \textit{weakly reciprocal} adapted processes as follows.
\begin{definition}\label{def:balancedasym}
 We say that an adapted stochastic matrix process $\Wc$ is \textit{weakly reciprocal} with coefficient $\al$ if for any regular adapted sequence $\Sc$, we have:
 \begin{align}\label{eqn:balancedasym}
\EXP{W_{\bar{S}(k+1)S(k)}(k+1)\mid \F_k}\geq \al \EXP{W_{S(k+1)\bar{S}(k)}(k+1)\mid \F_k}
 \end{align}
 for some $\al>0$.
\end{definition}
 Note that the restriction of the weakly reciprocal processes to deterministic chains of stochastic matrices is equivalent to the notion of balanced asymmetric chains as defined in \cite{bolouki2012}.

  Also note that if $\Sc$ is a regular adapted process, $\{\bar{S}(k)\}$ is also a regular adapted process which implies $\al\leq 1$ in \eqref{def:balancedasym}.

  The major challenging step towards proving Theorem~\ref{thrm:main} is to show that the balanced processes described in the statement of the result are weakly reciprocal.
  \begin{proposition}\label{prop:balancedspecial}
    Let $\Wc$ be a balanced process (see \eqref{eqn:stationarybalanced}) with coefficient $a\in(0,1]$ such that $W_{ii}(k)\geq \gamma$ almost surely for some $\gamma\in(0,1]$ and all $i\in [m]$ and  $k\geq 0$. Then $\Wc$ is weakly reciprocal with coefficient $\al=\frac{\gamma a}{4m}$.
  \end{proposition}
  \begin{proof}
   Suppose that the assumptions of the assertion hold and let $\Sc$ be an arbitrary adapted regular sequence. Fix $k\geq 0$. \beh{The strategy to prove the assertion is to partition the probability space into two events: $\Omega_g$, where $S(k+1)$ and $S(k)$ are \textit{roughly} the same, and its complement, $\bar{\Omega}_g=\Omega\setminus \Omega_g$. If $S(k+1)$ and $S(k)$ are roughly the same, then by the balanced assumption on the chain, we show that the weakly reciprocal condition follows. If $S(k)$ and $S(k+1)$ are substantially different, then the condition $W_{ii}(k)\geq \gamma$ will help us to show the weakly reciprocal condition. }

   Consider the measurable set $\Omega_g\in\F_k$, defined by
   \begin{align*}
     \Omega_g=&\bigg\{\omega\mid \EXP{\bfo_{S(k+1)\not=S(k)}\mid \F_k}
   < \frac{1}{2m}\EXP{W_{S(k+1)\bar{S}(k)}(k+1)\mid \F_k}\bigg\}.
   \end{align*}
    Then, we have:
    \begin{align}\label{eqn:proofpractbase}
     \EXP{W_{\bar{S}(k+1)S(k)}(k+1)\mid \F_k}
     =\EXP{W_{\bar{S}(k+1){S}(k)}(k+1)\mid \F_k}\bfo_{\Omega_g}
     +\EXP{W_{\bar{S}(k+1){S}(k)}(k+1)\mid \F_k}\bfo_{\bar{\Omega}_g},
    \end{align}
    \beh{where $\bar{\Omega}_g=\Omega\setminus\Omega_g$ is the complement of the set $\Omega_g$.}
   Note that
   \begin{align}\label{eqn:2cases}
    W_{\bar{S}(k+1){S}(k)}(k+1)=W_{\bar{S}(k+1){S}(k)}(k+1)(\bfo_{S(k+1)=S(k)}+\bfo_{S(k+1)\not=S(k)}).
    \end{align}
   If $S(k+1)\not=S(k)$ for some $\omega\in \Omega$, then for $i\in S(k+1,\omega)\setminus S(k,\omega)$, we have $W_{\bar{S}(k+1)S(k)}(k+1)\geq W_{ii}(k+1)\geq \gamma$ which follows from the assumptions of the proposition. Using this in \eqref{eqn:2cases}, we get:
   \begin{align}\label{eqn:subcase}
    \EXP{W_{\bar{S}(k+1){S}(k)}(k+1)\mid \F_k}&=\EXP{W_{\bar{S}(k+1){S}(k)}(k+1)(\bfo_{S(k+1)=S(k)}+\bfo_{S(k+1)\not=S(k)})\mid \F_k}\cr
   &\geq
   \EXP{W_{\bar{S}(k+1){S}(k)}(k+1)\bfo_{S(k+1)\not=S(k)}\mid \F_k}\cr 
   &\geq \gamma \EXP{\bfo_{S(k+1)\not=S(k)}\mid \F_k}.
   \end{align}
   Also on $\bar{\Omega}_g$, we have $\EXP{\bfo_{S(k+1)\not=S(k)}\mid \F_k}\geq \frac{1}{2m}\EXP{\Ws\mid \F_k}$. Therefore,
     \begin{align}\label{eqn:bomegasolved}
     \EXP{W_{\bar{S}(k+1)S(k)}(k+1)\mid \F_k}\geq \frac{\gamma}{2m}\EXP{\Ws\mid \F_k}\bfo_{\bar{\Omega}_g}+\EXP{W_{\bar{S}(k+1){S}(k)}(k+1)\mid \F_k}\bfo_{{\Omega}_g}.
   \end{align}

   So, it remains to analyze $\EXP{\Wbs\mid\F_k}$ on ${\Omega}_g$. Let us partition this event into two other events $\Omega_{ga}$ and $\Omega_{gb}=\Omega_{g}\setminus\Omega_{ga}$, where:
   \begin{align*}
     \Omega_{ga}&=\bigg\{\omega \in {\Omega}_g\mid \EXP{W_{S(k)\bar{S}(k)}(k+1)\F_k}\geq \frac{2m}{a}\EXP{\bfo_{\bar{S(k+1)\not=S(k)}}\mid \F_k}\bigg\}.
   \end{align*}

    Let us first analyze the expected flow over the event ${\Omega_{ga}}$. Note that $W_{\bar{S}(k+1)S(k)}(k+1),W_{\bar{S}(k)S(k)}(k+1)\in [0,m]$ almost surely. Using this, we have:
    \begin{align}\label{eqn:onga}
      &\EXP{W_{\bar{S}(k+1)S(k)}(k+1)\mid \F_k}\bfo_{\Omega_{ga}}\geq
      \EXP{W_{\bar{S}(k+1)S(k)}(k+1)\bfo_{S(k+1)=S(k)}\mid \F_k}\bfo_{\Omega_{ga}}\cr
      &\qquad=\EXP{W_{\bar{S}(k)S(k)}(k+1)\bfo_{S(k+1)=S(k)}\mid \F_k}\bfo_{\Omega_{ga}}
      \cr
      &\qquad =\bigg(\EXP{W_{\bar{S}(k)S(k)}(k+1)\mid \F_k}-\EXP{W_{\bar{S}(k)S(k)}(k+1)\bfo_{S(k+1)\not=S(k)}\mid \F_k}\bigg)\bfo_{\Omega_{ga}}\cr
      &\qquad\geq \bigg(\EXP{W_{\bar{S}(k)S(k)}(k+1)\mid \F_k}-m\EXP{\bfo_{S(k+1)\not=S(k)}\mid \F_k}\bigg)\bfo_{\Omega_{ga}}.
    \end{align}
    Now, note that $S(k)$ is measurable with respect to $\F_k$ and hence, we have:
    \begin{align}\nonumber
      \EXP{W_{\bar{S}(k)S(k)}(k+1)\mid \F_k}&=\sum_{\substack{S\subset[m]\\ |S|=|S(k)|}}
    \bfo_{S(k)=S}\EXP{W_{\bar{S}S}(k+1)\mid \F_k}\cr
    &\geq a\sum_{\substack{S\subset[m]\\ |S|=|S(k)|}}
    \bfo_{S(k)=S}\EXP{W_{S\bar{S}}(k+1)\mid \F_k}\cr
    &=\EXP{W_{S(k)\bar{S}(k)}(k+1)\mid \F_k}.
    \end{align}
    Replacing this equality in \eqref{eqn:onga}, we have:
    \begin{align}\nonumber
      &\EXP{W_{\bar{S}(k+1)S(k)}(k+1)\mid \F_k}\bfo_{\Omega_{ga}}\cr
      &\geq \bigg(a\EXP{W_{{S}(k)\bar{S}(k)}(k+1)\mid \F_k}-m\EXP{\bfo_{S(k+1)\not=S(k)}\mid \F_k}\bigg)\bfo_{\Omega_{ga}}.
    \end{align}
    But on $\Omega_{ga}$, we have $\EXP{W_{S(k)\bar{S}(k)}(k+1)\F_k}\geq \frac{2m}{a}\EXP{\bfo_{\bar{S(k+1)\not=S(k)}}\mid \F_k}$. Therefore,
    \begin{align}\label{eqn:ga}
      \EXP{W_{\bar{S}(k+1)S(k)}(k+1)\mid \F_k}\bfo_{\Omega_{ga}}\geq\frac{a}{2}\EXP{W_{{S}(k)\bar{S}(k)}(k+1)\mid \F_k}\bfo_{\Omega_{ga}}.
    \end{align}

     On $\Omega_{gb}$, we have:
    \begin{align}\label{eqn:gb}
    \EXP{\Wbs\mid \F_k}\bfo_{\Omega_{gb}}&\geq \EXP{\Wbs\bfo_{S(k+1)\not=S(k)}\mid \F_k}\bfo_{\Omega_{gb}}\cr
 &\geq\gamma\EXP{\bfo_{S(k+1)\not=S(k)}\mid \F_k}\bfo_{\Omega_{gb}}\geq
    \gamma\EXP{W_{S(k)\bar{S}(k)}(k+1)\mid \F_k}\bfo_{\Omega_{gb}}\cr
    &\geq \frac{\gamma a}{2m} \EXP{W_{S(k)\bar{S}(k)}(k+1)\mid \F_k}\bfo_{\Omega_{gb}}.
    \end{align}

     Combining \eqref{eqn:ga} and \eqref{eqn:gb}, we conclude that:
    \begin{align}\label{eqn:gintermediate}
    \EXP{\Wbs\mid \F_k}\bfo_{\Omega_{g}}\geq \frac{\gamma a}{2m} \EXP{W_{S(k)\bar{S}(k)}(k+1)\mid \F_k}\bfo_{\Omega_{g}}.
    \end{align}

    The next step is to relate $\EXP{W_{{S}(k)\bar{S}(k)}(k+1)\mid \F_k}\bfo_{\Omega_{g}}$ to $\EXP{W_{{S}(k)\bar{S}(k+1)}(k+1)\mid \F_k}\bfo_{\Omega_{g}}$.
     For this, we have:
    \begin{align}\nonumber
    &\EXP{W_{{S}(k)\bar{S}(k)}(k+1)\mid \F_k}\bfo_{\Omega_{g}}=
    \EXP{W_{{S}(k)\bar{S}(k)}(k+1)(\bfo_{S(k+1)=S(k)}+\bfo_{S(k+1)\not=S(k)})\mid \F_k}\bfo_{\Omega_{g}}\cr
    &\qquad\geq \EXP{W_{{S}(k)\bar{S}(k)}(k+1)\bfo_{S(k+1)=S(k)}\mid \F_k}\bfo_{\Omega_{g}}\cr
    &\qquad=
    \EXP{W_{S(k+1)\bar{S}(k)}(k+1)\mid \F_k}\bfo_{\Omega_{g}}-\EXP{W_{S(k+1)\bar{S}(k)}(k+1)\bfo_{S(k+1)\not=S(k)}\mid \F_k}\bfo_{\Omega_{g}}\cr
    &\qquad\geq  \bigg(\EXP{W_{S(k+1)\bar{S}(k)}(k+1)\mid \F_k}-m\EXP{\bfo_{S(k+1)\not=S(k)}\mid \F_k}\bigg)\bfo_{\Omega_{g}}.
    \end{align}
    But on $\Omega_{g}$ we have $\EXP{\bfo_{S(k+1)\not=S(k)}\mid\F_k}\leq \frac{1}{2m}\EXP{W_{S(k+1)\bar{S}(k)}(k+1)\mid \F_k}$. Therefore,
    \begin{align}\label{eqn:Srelate}
    \EXP{W_{{S}(k)\bar{S}(k)}(k+1)\mid \F_k}\bfo_{\Omega_{g}}\geq\frac{1}{2}\EXP{W_{S(k+1)\bar{S}(k)}(k+1)\mid \F_k}\bfo_{\Omega_{g}}.
    \end{align}

    Combining \eqref{eqn:gintermediate} and \eqref{eqn:Srelate}, we conclude that:
    \begin{align}\label{eqn:g}
    \EXP{W_{{S}(k)\bar{S}(k)}(k+1)\mid \F_k}\bfo_{\Omega_{g}}\geq\frac{1}{2}\EXP{W_{S(k+1)\bar{S}(k)}(k+1)\mid \F_k}\bfo_{\Omega_{g}}.
    \end{align}

     Replacing \eqref{eqn:g} in \eqref{eqn:gintermediate} and using \eqref{eqn:bomegasolved}, we finally find:
   \begin{align}\nonumber
     &\EXP{W_{\bar{S}(k+1)S(k)}(k+1)\mid \F_k}\cr
     &\qquad\geq \frac{\gamma}{2m}\EXP{\Ws\mid \F_k}\bfo_{\bar{\Omega}_g}
     +\frac{\gamma a}{4m} \EXP{W_{S(k+1)\bar{S}(k)}(k+1)\mid \F_k}\bfo_{\Omega_{g}} \cr
     &\qquad\geq \frac{\gamma a}{4m}\EXP{W_{{S}(k+1){\bar{S}}(k)}(k+1)\mid \F_k}.
   \end{align}
  \end{proof}

   The next step is to show that any random averaging dynamics generated by weakly reciprocal  adapted process is convergent up to a random permutation. \beh{More precisely, let an ordering of a vector $x\in\R^m$ be a vector $z\in\R^m$ such that $z_i=x_{\pi(i)}$ for all $i\in[m]$, where $\pi:[m]\to[m]$ is a permutation on $[m]$, and $z_1\leq z_2\leq\cdots\leq z_m$. Similarly, for a random vector $x:\Omega\to\R^m$, we say a random vector $z:\Omega\to\R^m$ is an ordering of $x$ if $z(\omega)$ is an ordering of $x(\omega)$ for (almost) all $\omega\in \Omega$. Then, we show that if $\xc$ is generated by a weakly reciprocal matrix process $\Wc$, its ordering converges almost surely (although $\xc$ itself may not be convergent).} The proof technique is based on the proof technique in \cite{hendrickx2011} and \cite{bolouki2012}, and the developed machinery above.
  \begin{proposition}\label{prop:mainaymmetric}
    Let $\Wc$ be an adapted stochastic matrix process that is weakly reciprocal  with coefficient $\alpha$ and let $\xc$ be a dynamics generated by $\Wc$.
    \begin{enumerate}[a. ]
       \item Let $z(k)$ be an ordering of $x(k)$. Then, $\lim_{k\to \infty}z(k)=z(\infty)$ exists almost surely.
       \item Consider the infinite flow event
    \begin{align}\nonumber
         &\Omega^{\infty}=\{\omega\in\Omega\mid \sum_{k=1}^{\infty}W_{\bar{S}(k+1)S(k)}(k)=\infty\cr
         &\qquad \mbox{for any adapted regular sequence $\{S(k)\}$}\}.
       \end{align}
       Then, \beh{for almost all $\omega\in\Omega^{\infty}$,} we have $\lim_{k\to\infty}(z_i(k)-z_j(k))=0$, i.e.\ agents reach consensus. As a result, on $ \Omega^{\infty}$, we almost surely have $\lim_{k\to\infty}x(k)=c \bfo$ for a random variable $c$.
       \item Suppose that $W_{ii}(k)\geq \gamma>0$ almost surely for all $i\in[m]$ and $k\geq 0$. Then, $\lim_{k\to\infty}x(k)$ exists almost surely.
    \end{enumerate}
  \end{proposition}
  \begin{proof}
  \begin{enumerate}[a. ]
    \item \beh{Fix an $\ell\in[m]$. Let $S_{\ell}(k):\Omega\to\Pset([m])$ be the index of the lower $\ell$ entries of $x(k)$, i.e.\ $S_{\ell}(k)$ is a random subset of $[m]$ such that (i) $|S_{\ell}(k)|=\ell$, and (ii) for almost all $\omega\in \Omega$, and for all $i\in S_{\ell}(k,\omega)$ and  $j\in \bar{S}_{\ell}(k,\omega)$, we have $x_i(k,\omega)\leq x_{j}(k,\omega)$.} Note that $S(k)$ is measurable with respect to $\F_k$. Now, let \[V_{\ell}(k)=\sum_{i=1}^{\ell}\beta^{i}z_{i}(k)=\sum_{i\in S(k)}\beta^{\pi^{-1}(i)}x_{i}(k),\]
    where $\beta=\frac{\al}{2}$. Note that $V_{\ell}(k)$ is measurable with respect to $\F_k$ \beh{and also, since $\{z_1(k)\}$ is an increasing sequence and $\{z_m(k)\}$ is a decreasing sequence almost surely (see e.g.\ \cite{SenetaCons}),
    \[mz_1(0)\leq m z_1(k)\leq V_{\ell}(k)\leq mz_m(k)\leq mz_m(0),\]
    and therefore, $|V_\ell(k)|\leq \|z(0)\|_\infty$ and as a result $\{V_{\ell}(k)\}$ is bounded almost surely. }Using some algebraic steps and the fact that $W(k)$ is stochastic almost surely, as shown in Eq.\ (31) and Eq.\ (32) in  \cite{bolouki2012}, it follows that almost surely:
    \begin{align}\nonumber
      V_{\ell}(k+1)-V_{\ell}(k)\geq \sum_{p=1}^{\ell-1}\bigg(\beta^{p}W_{S_p(k+1)\bar{S}_p(k)}(k+1)-\al^{p+1}W_{\bar{S}_p(k+1){S}_p(k)}(k+1)\bigg)
      \Delta z_p(k),
    \end{align}
    where $\Delta z_p(k)=z_{p+1}(k)-z_p(k)$. Applying conditional expectation on both sides of the above inequality and using the weakly reciprocal  property of $\Wc$, it follows that:
    \begin{align}\label{eqn:martingale}
      &\mathsf{E}[V_{\ell}(k+1)-V_{\ell}(k)\mid \F_k]\geq\sum_{p=1}^{\ell-1}\mathsf{E}[\beta^{p}W_{S_p(k+1)\bar{S}_p(k)}(k+1)-\beta^{p+1}W_{\bar{S}_p(k+1){S}_p(k)}(k+1)\mid \F_k]
      \Delta z_p(k)\cr
      &\geq \sum_{p=1}^{\ell-1}\mathsf{E}\bigg[2\beta^{p+1}W_{\bar{S}_p(k+1){S}_p(k)}(k+1)-\beta^{p+1}W_{\bar{S}_p(k+1){S}_p(k)}(k+1)\mid \F_k\bigg]
      \Delta z_p(k)\cr
      &=\sum_{p=1}^{\ell-1}\EXP{\beta^{p+1}W_{\bar{S}_p(k+1){S}_p(k)}(k+1)\mid \F_k}\Delta z_p(k).
    \end{align}
    From Doobs's Martingale Convergence Theorem (Theorem (2.10) \cite{Durrett}), one can immediately see that $V_{\ell}(k)$ is convergent almost surely for any $\ell\in[m]$. \beh{Finally, since $z_1(k)=\beta^{-1}V_1(k)$ and $z_{\ell+1}(k)=\beta^{-\ell-1}(V_{\ell+1}(k)-V_{\ell}(k))$,} it follows that $\lim_{k\to\infty}z(k)=z(\infty)$ exists almost surely.

    \item Since the process $\{V_{\ell}(k)\}$ is bounded almost surely, from \eqref{eqn:martingale}, it follows that:
    \begin{align}\nonumber
      &\sum_{k=0}^{\infty}\sum_{\ell=1}^{m-1}\sum_{p=1}^{\ell-1}\EXP{\beta^{p+1}W_{\bar{S}_p(k+1){S}_p(k)}(k+1)\mid \F_k}
      \Delta z_p(k)<\infty.
    \end{align}
     But since $\beta^{p+1}W_{\bar{S}_p(k+1){S}_p(k)}(k+1)\Delta z_p(k)\leq d(x(0))$ is bounded almost surely, from the dominated convergence theorem for conditional expectations (\cite{Durrett}, page 262), it follows that
     \begin{align*}
       &\mathsf{E}\bigg[\sum_{\substack{k\geq 0\\\ell \in [m-1]}}\sum_{p=1}^{\ell-1}{\beta^{p+1}W_{\bar{S}_p(k+1){S}_p(k)}(k+1)}
      \Delta z_p(k)\bigg]<\infty
     \end{align*}
    and hence, we almost surely have:
    \begin{align}\nonumber
      &\sum_{k=0}^{\infty}\sum_{\ell=1}^{m-1}\sum_{p=1}^{\ell-1}\beta^{p+1}W_{\bar{S}_p(k+1){S}_p(k)}(k+1)\Delta z_p(k)<\infty.
    \end{align}
    Now, if for some $\omega\in \Omega^\infty$, and some $i\in[m]$, we have $\lim_{k\to\infty}(z_i(k)-z_{i-1}(k))=z_i(\infty)-z_{i-1}(\infty)>0$, then since $\sum_{k=1}^\infty W_{\bar{S}_i(k+1){S}_{i-1}(k)}(k+1)=\infty$ on $\Omega^{\infty}$, it follows that $\sum_{k=1}^\infty W_{\bar{S}_i(k+1){S}_{i-1}(k)}(k+1)(z_{i}(k)-z_i(k))=\infty$. But since  $\sum_{k=0}^{\infty}\sum_{\ell=1}^{m-1}\sum_{p=1}^{\ell-1}\beta^{p+1}W_{\bar{S}_p(k+1){S}_p(k)}(k+1)
      \Delta z_p(k)<\infty$ it follows that for almost all points in $\Omega^{\infty}$, we have $z_{i}(\infty)-z_{i-1}(\infty)=0$. \beh{Therefore,  $\lim_{k\to\infty}(x_i(k)-x_j(k))=0$ for almost all $\omega\in\Omega^{\infty}$ and for all $i,j\in [m]$ which by Theorem~1 in \cite{SenetaCons} implies that $\lim_{k\to\infty}x(k)=c(\omega) \bfo$ for some $c(\omega)\in \R$ and almost all $\omega\in \Omega^{\infty}$.}

      \item Suppose that for all $k\geq 0$ and $i\in[m]$, $W_{ii}(k)\geq \gamma$ almost surely and suppose that on a set $\Omega'\subset \Omega$, $\lim_{k\to\infty}x(k)$ does not exist. Without loss of generality, we may assume that there exists $i\in [m]$ such that $\lim_{k\to\infty}x_i(k)$ does not exists on the set $\Omega'$ (otherwise, we can restrict our discussion to such a set). First notice that $\lim_{k\to\infty}z(k)\not=c\bfo$ on $\Omega'$, otherwise, as in the previous case, this implies that $\lim_{k\to\infty}x(k)=c\bfo$. Now, fix an $\omega\in\Omega'$ and consider the corresponding sample path of the dynamics. Let $\{a_1,\ldots,a_q\}=\{z_1(\infty),\ldots,z_m(\infty)\}$ with $a_1< \ldots< a_q$ be the distinct values of the entries of $z(\infty)$ for the sample point $\omega$ ($q\leq m$).  Note that, for any $\epsilon\leq\frac{1}{4}\min_{1\leq p<q}(a_{p+1}-a_{p})$, there exists a time instance $T_\epsilon\geq 0$ such that for $k\geq T_\epsilon$, $x_i(k)$ is at the $\epsilon$-neighborhood of one of the points in $\{a_1,\ldots,a_q\}$. This point is unique because $\epsilon\leq\frac{1}{4}\min_{1\leq p<q}(a_{p+1}-a_{p})$. Let the index of that point be $p(k)$, i.e.\ $|x_i(k)-a_{p(k)}|<\epsilon$ for $k>T_\epsilon$.   But since $\lim_{k\to\infty} x_i(k)$ does not exists, it follows that there is a sequence of the increasing time instances $k_1<k_2<\ldots$ such that $p(k_t)\not= p(k_{t}+1)$. This implies that $S(k_t+1)\not= S(k_t)$ for some $\ell$, as defined in part a. and also, $z_{\ell+1}(k_t)-z_{\ell}(k_t)\geq \frac{1}{4}\min_{1\leq p<q}(a_{p+1}-a_{p})$. But since $W_{ii}(k_t+1)\geq \gamma$ almost sure for all $i$, it follows that $W_{\bar{S}(k_t+1)S(k_t)}(k_t+1)\geq \gamma$, and hence,
      \begin{align*}
         &\sum_{k=0}^{\infty}\sum_{\ell=1}^{m-1}\sum_{p=1}^{\ell-1}\beta^{p+1}W_{\bar{S}_p(k+1){S}_p(k)}(k+1)
      \Delta z_p(k)=\infty,
      \end{align*}
      for the sample point $\omega\in \Omega'$. But based on part b. this happens almost never, and hence, it follows that $\prob{\Omega'}=0$ and hence, \beh{}.
  \end{enumerate}
   \end{proof}

 Theorem~\ref{thrm:main} directly follows from a combination of Proposition~\ref{prop:balancedspecial} and Proposition~\ref{prop:mainaymmetric}: by Proposition~\ref{prop:balancedspecial}, balanced random matrix processes $\Wc$ with the property that $W_{ii}(k)\geq \gamma$ almost surely for all $k\geq 0$ and $i\in [m]$ are weakly reciprocal and based on part c.\ of Proposition~\ref{prop:mainaymmetric}, it follows that the dynamics system \eqref{eqn:dynsys} is almost surely convergent for such random matrix processes.

 Another immediate consequence of the above theorem is that the evaluation of any symmetric and continuous function along the trajectories of the random dynamics generated by adapted weakly reciprocal   process is convergent. More precisely, let $\sigma:[m]\to[m]$ be an arbitrary permutation over the set $[m]$. For a vector $x\in \R^m$, let $y=x_\sigma$ be the vector defined by $y_i=x_{\sigma(i)}$. A function $V:\R^n\to\R$ is said to be symmetric if $V(x_\sigma)=V(x)$ for any $x\in \R^m$ and permutation $\sigma$. An example of a symmetric function is $V(x)=\sum_{i=1}^mx_i$.
 \begin{corollary}\label{cor:symmetric}
   Let $\xc$ be a dynamics generated by a weakly reciprocal  adapted process $\Wc$ and let $V:\R^n\to\R$ be a continuous symmetric function. Then, $\lim_{k\to\infty}V(x(k))$ exists almost surely.
 \end{corollary}
 \begin{proof}
   By Proposition~\ref{prop:mainaymmetric}, $\lim_{k\to\infty} z(k)$ exists, where $z(k)$ is an ordering of $x(k)$. Since $V(\cdot)$ is symmetric, it follows that $V(x(k))=V(z(k))$ and hence,
   \[\lim_{k\to\infty}V(x(k))=\lim_{k\to\infty}V(z(k))=V(\lim_{k\to\infty}z(k)),\]
   where the last equality follows from the continuity of $V(\cdot)$.
 \end{proof}

 Now consider any initial condition $x(t_0)=e_i$ (more precisely, $x(t_0,\omega)=e_i$)) where $t_0\geq 0$ is an arbitrary starting time. Applying Corollary~\ref{cor:symmetric} to function $V(x)=\frac{1}{m}1^Tx$, we conclude that for any $t_0\geq 0$, the random vector:
 \[\pi(k)=\lim_{k\to\infty}\frac{1}{m}e^TW(k)\cdots W(k),\]
 is well-defined. Also, note that we almost surely have $\pi^T(k+1)W(k+1)=\pi^T(k)$ for any $k\geq 0$. Thus, if we define
\begin{align}\label{eqn:defabsolute}
  \bar{\pi}(k)=\EXP{\pi(k)\mid \F_k},
\end{align}
the following result follows immediately.
\begin{corollary}\label{cor:absoluteinf}
  Any weakly reciprocal  adapted stochastic matrix process $\Wc$ admits an adapted absolute probability process (as defined in \cite{TouriNedich:productrandom}), i.e.\ an adapted random vector process $\{\bar{\pi}(k)\}$ such that for any $k\geq 0$, we have:
  \[\EXP{\bar{\pi}^T(k+1)W(k+1)\mid \F_k}=\bar{\pi}^T(k).\]
\end{corollary}
\begin{proof}
 Let $\{\bar{\pi}(k)\}$ be the vector process defined by \eqref{eqn:defabsolute}. Then, we have
 \begin{align*}
   &\EXP{\bar{\pi}^T(k+1)W(k+1)\mid \F_k}=\EXP{\EXP{\pi^T(k+1)\mid \F_{k+1}}W(k+1)\mid \F_k}\cr
   &\qquad=\EXP{\pi^T(k+1)W(k+1)\mid \F_k}=\EXP{\pi^T(k)\mid \F_k}=\bar{\pi}^T(k).
 \end{align*}
\end{proof}
 Since any asymmetric balanced chain admits an adapted absolute probability process, by Theorem~2 in \cite{TouriNedich:productrandom}, Corollary~\ref{lemma:mainlyapunov} follows.

\section{Implications}\label{sec:implicaitons}
 In this section, we revisit the motivational problems mentioned in Section~\ref{sec:motivation}. We first revisit the random variations of Hegselmann-Krause dynamics and then we discuss the endogenous gossiping dynamics and how Theorem~\ref{thrm:main} can be used to study them.

 \subsection{Hegselmann-Krause Dynamics and Asymmetric Endogenous Gossiping}
  It is not hard to see that all the random instances of the Hegselmann-Krause dynamics discussed in Section~\ref{sec:motivation} are examples of the dynamics \eqref{eqn:dynsys}. For example for the case of the asynchronous Hegselmann-Krause dynamics, suppose that $i(k)$ is a random agent picked by nature at time $k\geq 0$. Then, we have:
 \begin{align*}
   W_{i(k)j}(k+1)\!=\!\left\{\begin{array}{ll}
     \!\frac{1}{|\N_{i(k)}(x(k),\epsilon)|}&\mbox{if $j\in \N_{i(k)}(x(k),\epsilon)$}\\
     \!0& \mbox{otherwise}
   \end{array}\right.
 \end{align*}
 where $\epsilon$ is the confidence level of agents. Note that if the process $\{i(k)\}$ is an adapted process and if we also have $\prob{i(k)=\ell\mid \F_k}\geq \underline{p}$ for any $\ell\in [m]$, then for any $i,j\in [m]$ we have
 \[\EXP{W_{ij}(k+1)\mid \F_k}\geq \frac{\underline{p}}{m}\EXP{W_{ji}(k+1)\mid \F_k}.\]

 It is not hard to see that a similar condition holds for the other random instances of the Hegselmann-Krause dynamics proposed in Section~\ref{sec:motivation}, i.e.
 \[\EXP{W_{ij}(k+1)\mid \F_k}\geq \eta \EXP{W_{ji}(k+1)\mid \F_k},\]
 for some $\eta>0$. Also, in all of those models, $W_{ii}(k)\geq \frac{1}{m}$ for any $i\in [m]$ and $k\geq 0$.

  Similarly, the asymmetric endogenous gossiping dynamics \eqref{eqn:gossipextended} is another example of the dynamics \eqref{eqn:dynsys}. In this case, for $\ell\not=j(k)$, we have $W_{\ell\ell}(k)=1$, $W_{j(k)j(k)}(k+1)=1-\gamma(k)$, $W_{j(k)i(k)}(k+1)=\gamma(k)$ and the rest of the entries are zero. In this case, if $\alpha(k)$ is independent of choice of $(i(k),j(k))$ and \eqref{eqn:gossipprob} also holds, then we have:
 \begin{align}\nonumber
   \EXP{W_{ij}(k+1)\mid \F_k}\geq \alpha l \EXP{W_{ji}(k+1)\mid \F_k}.
 \end{align}
  Also note that in this case, we have $W_{ii}(k)\geq 1-h>0$ for all $i\in [m]$ and $k\geq 0$.

  Note that both the Hegselmann-Krause dynamics and the endogenous gossiping dynamics share the common property of
  \begin{align}\label{eqn:mutual}
    \EXP{W_{ij}(k+1)\mid \F_k}\geq \eta \EXP{W_{ji}(k+1)\mid \F_k},
  \end{align}
  for some $\eta>0$. Following the terminology in \cite{Boluki11}, we say that $\Wc$ has adapted sub-symmetric property if it satisfies \eqref{eqn:mutual}. In fact, this property insures the balanced property, as for any non-trivial $S\subset [m]$, we have:
  \begin{align}\nonumber
    \EXP{W_{S\bar{S}}(k+1)\mid \F_k}&=\sum_{i\in S}\sum_{j\in \bar{S}}\EXP{W_{ij}(k+1)\mid \F_k}\cr
    &\geq \eta \sum_{i\in S}\sum_{j\in \bar{S}}\EXP{W_{ji}(k+1)\mid \F_k}\cr
    &=\eta\EXP{W_{\bar{S}S}(k+1) \mid \F_k}
  \end{align}

  Therefore, we have the following corollary.
  \begin{corollary}
    Let $\Wc$ be an adapted sub-symmetric matrix process and $W_{ii}(k)\geq \gamma$ almost surely for all $i\in [m]$ and $k\geq 0$. Then any dynamics $\xc$ generated by $\Wc$ is convergent.
  \end{corollary}
 As a result of the above corollary, the various random instances of the Hegselmann-Krause dynamics as well as asymmetric endogenous gossiping dynamics satisfying \eqref{eqn:gossipprob} is convergent almost surely.
\\{\textbf{Acknowledgement. }}
  We would like to thank anonymous reviewers for the valuable comments and suggestions for the improvement of this work.

\section{Conclusion}\label{sec:conc}
In this work, we have studied averaging dynamics driven by random adapted stochastic matrix processes. We showed that under so-called balanced conditions and strictly positive diagonal entries of the underlying matrix process, such dynamics converge almost surely. Our proof relies on various properties of novel objects, \textit{weakly reciprocal matrix processes}, and their connection to balanced processes with strictly positive diagonal entries. We also showed that those dynamics admit infinitely many (stochastic) Lyapunov functions which open the door to rate of convergence analysis of the corresponding averaging dynamics. Using our main results, we showed that asynchronous Hegselmann-Krause dynamics, Hegselmann-Krause dynamics with link failure, and endogenous asymmetric gossip algorithms converge almost surely.

 We believe that the application domain of the results and tools developed in this work goes beyond the few examples discussed here. Applications of these results in distributed optimization in endogenously changing environment, Markov-chains in random environments, convergence rate analysis of  consensus dynamics in random environments, and distributed learning are remained to be explored in future works.
\bibliographystyle{ieeetr}
\bibliography{color}
\end{document}